\documentclass[11pt]{amsart}

\usepackage{amsmath}
\usepackage{amssymb}

\usepackage{amsfonts}
\usepackage{amsthm}
\usepackage{enumerate}
\usepackage[mathscr]{eucal}

\usepackage{graphicx}
\usepackage{verbatim}

\newcommand{\vbf}{\bf}

\newcommand{\ci}[1]{_{{}_{\!\scriptstyle{#1}}}}

\newcommand{\Be}{\begin{equation}}
\newcommand{\Ee}{\end{equation}}

\newcommand{\Bea}{\begin{eqnarray}}
\newcommand{\Eea}{\end{eqnarray}}
\newcommand{\Beas}{\begin{eqnarray*}}
\newcommand{\Eeas}{\end{eqnarray*}}
\newcommand{\Benu}{\begin{enumerate}}
\newcommand{\Eenu}{\end{enumerate}}
\newcommand{\Bi}{\begin{itemize}}
\newcommand{\Ei}{\end{itemize}}
\def\hz#1#2{{\dot \sK_2^{#1,#2}}}
\def\hzr#1#2{{\dot \sK_r^{#1,#2}}}

\def\a{\alpha}

\def\h{h}

\def\intslash{\rlap{\kern  .32em $\mspace {.5mu}\backslash$ }\int}
\def\qsl{{\rlap{\kern  .32em $\mspace {.5mu}\backslash$ }\int_{Q_x}}}

\def\vth{\vartheta}

\def\emph#1{{\it #1 }}

\def\ga{\gamma}

\def\eg{{\it e.g. }}
\def\cf{{\it cf}}

\def\supp{{\text{\rm supp}}}

\def\inn#1#2{\langle#1,#2\rangle}

\def\noi{\noindent}

\def\meas{{\text{\rm meas}}}

\def\lc{\lesssim}
\def\gc{\gtrsim}

\def\eps{\varepsilon}

\def\ka{\kappa}
            
\def\la{\lambda}

\def\fA{{\mathfrak {A}}}

\def\fa{{\mathfrak {a}}}

\def\bbN{{\mathbb {N}}}

\def\bbR{{\mathbb {R}}}

\def\bbZ{{\mathbb {Z}}}

\def\sK{{\mathscr {K}}}
\def\sL{{\mathscr {L}}}

\def\cA{{\mathcal {A}}}

\def\cE{{\mathcal {E}}}
\def\cF{{\mathcal {F}}}

\def\cK{{\mathcal {K}}}

\def\cM{{\mathcal {M}}}
\def\cN{{\mathcal {N}}}

\def\cS{{\mathcal {S}}}

 at 10 true pt

\def\be#1{\begin{equation}\label{ #1}}
\def\endeq{\end{equation}}
\def\endal{\end{align}}
\def\bas{\begin{align*}}
\def\eas{\end{align*}}
\def\bi{\begin{itemize}}
\def\ei{\end{itemize}}

\def\eps{\varepsilon}

\def\emph#1{{\it #1}}
\def\textbf#1{{\bf #1}}

\theoremstyle{plain}
  \newtheorem{theorem}{Theorem}[section]

   \newtheorem{proposition}[theorem]{Proposition}
   
   \newtheorem{lemma}[theorem]{Lemma}
   \newtheorem{corollary}[theorem]{Corollary}

\theoremstyle{remark}

\theoremstyle{definition}

\begin{document}


\title[Radial multipliers and almost everywhere convergence]
{On radial Fourier multipliers and almost everywhere convergence}

\author{Sanghyuk Lee \ \ \  \  \ Andreas Seeger}
\address{Sanghyuk Lee\\ School of Mathematical Sciences, Seoul National University, Seoul 15\
1-742, Korea} \email{shklee@snu.ac.kr}

\address{Andreas Seeger \\ Department of Mathematics \\ University of Wisconsin \\480 Lincoln Drive\\ Madison, WI, 53706, USA} \email{seeger@math.wisc.edu}

\begin{abstract}
We study a.e.  convergence on $L^p$,
and  Lorentz spaces $L^{p,q}$,  $p>\tfrac{2d}{d-1}$,
for variants of Riesz means at the critical index $d(\tfrac 12-\tfrac 1p)-\tfrac12$.  We derive more general results for (quasi-)radial Fourier multipliers and associated maximal functions, acting  on $L^2$ spaces with power weights, and their interpolation spaces. We also include a characterization of boundedness of such  multiplier transformations on weighted $L^2$ spaces, and a sharp endpoint bound for Stein's square-function associated with the Riesz means.
\end{abstract}


\thanks{Supported in part by NRF grant 2012008373 and NSF grant 1200261}
\maketitle



\section{Introduction}
Let $\rho\in C^\infty(\bbR^d\setminus \{0\})$ be a  homogeneous distance  function of degree $\beta>0$, i.e.  $\rho$ satisfies
$\rho(t^{1/\beta}\xi)=t\rho(\xi)$ for all $t>0$ and  $\rho(\xi)>0$ for $\xi\neq 0$.
For Schwartz functions $f$  on $\bbR^d$ 
($d\ge 2$ throughout this paper)
 define the Riesz means  $S^\la_t f$ of the Fourier integral by
$$S^\la_tf(x)=\frac{1}{(2\pi)^d}\int_{\rho(\xi)\le t}\Big(1-\frac{\rho(\xi)}{t}\Big)^\la
\widehat f(\xi)e^{i\inn x\xi}\,d\xi.$$ 

In order to consider almost everywhere convergence on $L^p$ one needs to be able to define
$S^\la_t f$ as a measurable function, for all $f\in L^p(\bbR^d)$.
A necessary condition  for $S^\la_t$ to extend as a continuous
operator  from  $L^p$ to the space  $\cS'$ of tempered distributions
is  that the convolution kernel belongs to $L^{p'}$  (\cf.
\cite{caso} for a similar comment, and \S\ref{BRdef} below). This is
the case if and only if $\la>\la(p):=d(\frac 12-\frac1p)-\frac12$.
 In view of the compact support of the multiplier the
distribution $S^\la_t f$ is then a bounded  $C^\infty$ function; moreover the maximal  function $\sup_{t>0} |S^\la_t f(x)|$ is  well defined as a  measurable function.
 Carbery, Rubio de Francia and Vega \cite{crv} studied the pointwise  behavior of $S^\la_t$ and showed that, for $\rho(\xi)=|\xi|$ and  $\la>\max\{\la(p),0\}$,
the means  $S^\la_t f(x)$ converge almost everywhere to $f(x)$.
 The method in  \cite{crv}
is based on the trace theorem for Sobolev functions and
applies to the  general situation considered here, see also
 Sato \cite{sato}. For an extension involving  nonisotropic distance functions
see Cladek \cite{cladek}.


In this paper we investigate  what happens at the critical index
$\la=\la(p)=d(\frac12-\frac1p)-\frac12$ when $\la(p)>0$, {\it i.e.}
when $p>\frac{2d}{d-1}$. It is natural  either to consider $S^\la_t$
on smaller Lorentz spaces, or to slightly regularize the  multiplier
to produce an operator well defined on $L^p(\bbR^d)$.  Let
\Be\label{hlaga}h_{\la,\ga}(r)= \frac{(1-r)^\la_+}
{(1+\log(\tfrac{1}{1-r}))^\ga}\,.\Ee Define the generalized Bochner
Riesz means $S^{\la,\gamma}_tf$ by
$$\widehat {S^{\la,\gamma}_t\! f}(\xi)= h_{\la,\ga}(\tfrac{\rho(\xi)}{t}) \widehat f(\xi).$$
For the critical  $\la=\la(p)$
these operators
extend to  bounded operators from the Lorentz space $L^{p,q}$  to $\cS'$ if
and only if $\gamma>1-1/q$ when $q>1$ and $\gamma\ge 0$ when $q=1$;
\cf.  \S\ref{BRdef}.
In the thesis \cite{annoni} Annoni showed, for $\rho(\xi)=|\xi|$,  that
$S^{\la(p),\gamma}_t\! f(x)\to f(x)$ a.e. for all $f\in L^p$, $p>\frac{2d}{d-1}$
under the condition  $\gamma> 3/2-1/p$. This left  open the range
$1-1/p<\ga\le 3/2-1/p$. Almost everywhere convergence  in this range is implied
 by  the case $q=p$ of the following theorem.

\begin{theorem}\label{laga}
Let $\frac{2d}{d-1}<p<\infty$,
$\la(p)=d(\frac 12-\frac 1p)-\frac12$.
Let $\gamma>1-\frac 1q$, $1<q\le \infty$, or $\gamma\ge 0$ and $q=1$. Then for
$f\in L^{p,q}$ we have
$$\lim_{t\to\infty}  S^{\la(p),\gamma}_t f(x)=f(x) \text{ a.e. }$$
\end{theorem}

\noi Note that for $q=1$ this covers  an endpoint result for the
Riesz means.

The usual approach to  pointwise convergence
results for functions in $L^{p,q}(\bbR^d)$ is to prove
an $L^{p,q}\to L^{p,\infty}$ bound for the associated maximal function
$S^{\la,\gamma}_*\!f= \sup_{t>0}|S^{\la,\gamma}_t\!f|$. By Stein's theorem
\cite{steinlimits}
such an estimate is necessary for $p\le 2$, see Tao \cite{tao} for
what is currently known for the maximal Bochner-Riesz operator in that range.
For $p>2$ complete $L^p$-boundedness results for the maximal Bochner-Riesz operator ($\la>\la(p)$) are known in two dimensions (\cite{carbery})
and partial results have been proved in higher dimensions (\cite{christ}, \cite{seegerthesis}, \cite{lrs-proc}, \cite{leese14}).
For the endpoint $\la=\la(p)$ sharp results were recently observed in
\cite{lrs},  for the case $\rho(\xi)=|\xi|$,
in the restricted range
$\frac{2d+2}{d-1}<p<\infty$. For these parameters
$S^{\la(p),0}_*$
maps $L^{p,1} $ to $L^p$  and
$S^{\la(p),\gamma}_*$ maps $L^{p,q}$ to $L^p$ if
$q\le p$ and $\gamma>1-1/q$. We note that
the sharp $L^{p,1}\to L^p$ endpoint bounds for  the maximal operator are not known
in the range $\frac{2d}{d-1}< p\le \frac{2(d+1)}{d-1}$,  even in two dimensions. Still less is known for general $\rho$.
Thus we  opt for a variant of
  the approach by Carbery, Rubio de Francia and Vega \cite{crv}, who used
 weighted $L^2(|x|^{-a}dx)$ spaces.
Annoni \cite{annoni}  followed this approach and introduced
logarithmic  modifications of the weight function, working with
$|x|^{-a} (\log(2+ |x|))^{-\mu}$ for suitable $\mu>0$.
We prefer  to preserve the
homogeneity of the weight and use the observation that for $p>2$
 the space $L^{p,2}$ is embedded in $L^2(|x|^{-d(1-2/p)}dx)$. We are able to sharpen  the analysis in \cite{crv} to prove boundedness in this space
for   maximal operators defined by
 $\sup_{t>0}|\cF^{-1}[h(\tfrac{\rho(\cdot)}{t})\widehat f\,]|$ where $h$ is
supported in $[\tfrac 12,2]$ and belongs to the $L^2$-Sobolev space
$\sL^2_\alpha\equiv B^2_{\a,2}$ with $\alpha=d(1/2-1/p)$.
This will be  a special case of Theorem \ref{herzest}  below
and can also  be deduced  from the square-function estimate in
Theorem \ref{sq}.
In particular for the range  $\gamma>1/2$   the maximal operator
$S^{\la(p),\ga}_*$ is bounded on
$L^2(|x|^{-d(1-2/p)}dx)$; this implies the
$q=2$ case
 of Theorem \ref{laga}.

In order to get a complete result for all $q$
(in particular $q=p$)
it is convenient to work with
 the homogeneous Herz spaces
$\hzr{\ga}{r}$, in
the  special case $r=2$. For fixed $r$ these  are  real interpolation spaces  of the $L^r(|x|^{r\gamma}dx)$ spaces, see Gilbert \cite{gilbert}.
The definition (following  the terminology  in \cite{baernstein-sawyer})
is as follows.  Let, for $l\in \bbZ$,
$\fA_l:=\{x:2^l\le |x|<2^{l+1}\}$. Then  $\hzr{\ga}{q}$ is
the space of all functions which are $r$-integrable on compact subsets of $\bbR^d\setminus\{0\}$ and for which
\Be\label{herzdef}
\|f\|_{\hzr{\ga}{q}}=
\Big(\sum_{l\in \bbZ}2^{l\ga q} \Big[ \int_{\fA_l} |f(x)|^r\, dx\Big]^{q/r}\Big)^{1/q}
\Ee
is finite. It is easy to see that for $\gamma>-d/r$ every function in
$\hzr{\ga}{q}$ defines a unique tempered distribution on  $\cS'(\bbR^d)$.
In what follows we use the notation $\lc$ for an inequality which involves an implicit  constant.

\begin{theorem} \label{herzest}
Let $\frac12<\alpha<\frac d2$, $1\le q\le \infty$ and
$s=\frac{q}{q-1}$. Let $h\in B^2_{\alpha,s}(\bbR)$ be supported in
$(\frac 12, 2)$. Then the following hold.

(i) For $2\le q\le \infty$,
$$
\big\| \sup_{t>0} \big|\cF^{-1}[h(\tfrac{\rho(\cdot)}{t}) \widehat
f\,]
 \big\|_{\hz{-\alpha}{q}}
 \lc \|h\|_{B^2_{\alpha,s}}
\|f\|_{\hz{-\alpha}{q}}\,.
$$

(ii)  For $1\le q\le2$,
$$
\big\| \sup_{t>0} \big|\cF^{-1}[h(\tfrac{\rho(\cdot)}{t}) \widehat
f\,]
 \big\|_{\hz{-\alpha}{2}}
 \lc \|h\|_{B^2_{\alpha,s}}
\|f\|_{\hz{-\alpha}{q}}\,.
$$
\end{theorem}

For $p>2$ and $\alpha=d(1/2-1/p)$ there is the embedding
$L^{p,q}\subset \hz{-\alpha}{q}$, (\cf. Lemma \ref{Loremb} below).
By a standard  approximation argument we will show
\begin{corollary}\label{herzcor}
Let $p>\frac{2d}{d-1}$, $1\le q\le\infty$, $s=\frac{q}{q-1}$,
 and let $h\in B^2_{d(\frac 12-\frac1p),s}$
be supported in a compact subinterval of  $(0,\infty)$. Then
for  $f\in L^{p,q}$
$$\lim_{t\to\infty} |\cF^{-1}[h(\rho/t)\widehat f\,](x)|=0 \text{ a.e.}$$
\end{corollary}

If  $\chi\in C^\infty_0(\bbR)$ is  compactly supported away from the origin
then $\chi h_{\la,\ga}\in B^2_{\a,s}$ for $\a=\la+1/2$ and $\gamma>1/s$, and $\gamma\ge 0$ if $s=\infty$,  see \eg
\cite{taibleson}. Thus  one
can deduce  Theorem \ref{laga}   from Corollary \ref{herzcor}.

The proof of Theorem \ref{herzest} for $q=2$  also
gives a   characterization of
boundedness convolution operators with quasiradial multipliers on
weighted  $L^2$-spaces with power weights  (i.e. on the spaces
$\dot \sK^{2}_{\pm \alpha,2}$).
The result  is in the spirit of (but in some sense complementary to)
 the results  by
Muckenhoupt, Wheeden  and Young  \cite{mwy}.
 In what follows let $\varphi\in C^\infty_0(\tfrac 12, 2)$ be a nontrivial
 bump function.

\begin{theorem} \label{multch}
Let $\rho$ be as above and let $1/2<\alpha<d/2$. Define the operator $T$
by $\widehat {Tf}(\xi)= m(\rho(\xi))\widehat f(\xi)$. Then the following statements are equivalent:

(i) $T$  is bounded on $L^2(|x|^{-2\alpha} dx)$.

(ii)  $T$  is bounded on $L^2(|x|^{2\alpha} dx)$.

(iii) $\sup_{t>0}\|\varphi m(t\cdot)\|_{L^2_\alpha(\bbR)} <\infty$.
\end{theorem}

Finally we  prove a  sharp bound for Stein's square-function
 associated with the Riesz means,
$$G_\alpha f(x)= \Big(\int_0^\infty
\big| S^{\alpha-1}_t f(x)-S^\alpha_t f(x)\big|^2\frac{dt}{t}\Big)^{1/2}\,$$
 which can also be used to prove  Theorem \ref{multch}.

\begin{theorem}\label{sq}
Let $\frac 12<\alpha<\frac d2$.
Then for $f\in L^2(|x|^{-2\alpha}dx)$,
$$\int\big| G_\alpha f(x)\big|^2 \frac{dx}{|x|^{2\a}}\lc
\int \big|f(x)\big|^2 \frac{dx}{|x|^{2\a}}\,.$$
\end{theorem}

\medskip

\subsection*{\it This paper.} In \S\ref{BRdef} we discuss the definition
of our multiplier transformations and associated maximal functions
on Lebesgue and  Lorentz spaces, and then prove a convenient
characterization of quasiradial multipliers to be Fourier transforms
of functions in these spaces (see Theorem \ref{FLuprop}). In the
preliminary section \S\ref{prel} we consider embeddings for Lorentz
and Herz spaces which are needed to deduce Theorem \ref{laga} from
Corollary \ref{herzcor}. We also review the Fourier restriction
estimate in weighted $L^2$ spaces, use it to prove a basic maximal
function estimate, and discuss other basic preliminaries. The main
section \S\ref{herzestsect} contains  the proof of Theorem
\ref{herzest}. The proofs of Theorem \ref{multch} and Corollary
\ref{herzcor} are given in \S\ref{concl}, and  Theorem \ref{sq} is
proved in \S\ref{sqsect}.

\subsection*{\it Acknowledgement} The authors would like to thank Jong-Guk Bak for conversations
at the early stages of this work.

\section{Definition  of  convolution and maximal operators
 on $L^{p,q}$}\label{BRdef}
For $f\in L^p$ the expression  $S^{\la, \gamma}_t f$ is not necessarily
defined for all $f\in L^{p,q}$, even in the sense of tempered distributions, since
the Fourier  transform of an $L^p$ function does not have to be a function.

\begin{lemma} \label{TKlemma} Suppose $1\le p,q<\infty$. Let $K$ be a tempered distribution in $\bbR^d$  such that $\widehat K$ has compact support. For $f\in \cS(\bbR^d)$ define $T_K=f*K$. Then

(i)  $T_K$ extends to a continuous linear operator from $L^{p,q}$ to $\cS'$ if and only if $K\in L^{p',q'}$.

(ii) If $K\in L^{p',q'}$ then for $f\in L^{p,q}$ the maximal function
$$\sup_{t>0}|t^{d}K(t\cdot)* f(x)|$$   is Borel  measurable.
\end{lemma}
\begin{proof}
Assume that $T_K: L^{p,q}\to \cS'$ is continuous.
 Then for every $g\in \cS$ there
is a constant $C(g)$ such that $|\inn{T_K f}{g}|\le
C(g)\|f\|_{L^{p,q}}$. Observe that if $g\in \cS$ such that $\widehat
g(\xi)=1$ for $\xi\in \supp(\widehat K)$ then $\inn{K*f}{g}=
\inn{f}{K(-\cdot)}$. Thus $K$ must lie in the dual space of
$L^{p,q}$, i.e. in  $L^{p',q'}$. Conversely, if $K\in K^{p',q'}$
then for every $f\in L^{p,q}$ the convolution $x\mapsto K*f(x)$ is
well defined as a bounded continuous function. This shows $(i)$.
Since the supremum of continuous functions is Borel measurable we
get $(ii)$.
\end{proof}

\subsection*{Quasiradial functions  in $\cF L^{u,s}$}\label{quasiradial}
We consider $m=h\circ\rho$ where $\rho$ is a homogeneous  distance
function of degree $\beta$, and $\rho\in C^\infty(\bbR^d\setminus
\{0\})$. It is useful to express the condition $(ii)$ in Lemma
\ref{TKlemma} in terms of the one-dimensional Fourier transform of
$\h_\beta(t)= h(t^\beta)$. The following theorem sharpens a result
in \cite{se-toh} for the class of distance functions considered
here; the radial analogue, for $\rho(\xi)=|\xi|$, is already in
\cite{gs}. Note that here we make no curvature assumption  on the $\rho$-unit sphere
$$\Sigma_\rho=\{\xi:\rho(\xi)=1\}.$$  The following result may be interesting in its own right but we
shall use it only to demonstrate sharpness of our
results; it is not needed in the proofs of  Theorems
\ref{herzest}-\ref{sq}.

\begin{theorem} \label{FLuprop}Let $h$ be supported in a compact subinterval $J$  of $(0,\infty)$ and let $$\cK_{\beta}(r) = \frac{1}{2\pi}\int h(t^\beta) e^{i rt} dt\,.$$
Let  $1<u<2$, $1\le s\le \infty$ and  $\mu_d$ be the measure on $\bbR$ given by
$$d\mu_d(r)= (1+|r|)^{d-1} dr\,.$$ Then
\Be\label{flqeq}\big\|\cF^{-1} [h\!\circ\!\rho] \big\|_{L^{u,s}(\bbR^d)} \approx
\Big\| \frac{ \cK_{\beta } }{(1+|\cdot|)^{\frac{d-1}{2}}} \Big\|_{L^{u,s}(\bbR, \mu_d)}.\Ee
\end{theorem}
For $u=s$ this equivalence becomes
$$\big\|\cF^{-1} [h\!\circ\!\rho] \big\|_{L^{u}(\bbR^d)} \approx
\Big(\int |\cK_{\beta}(r)|^u (1+|r|)^{(d-1)(1-\frac u2)}
dr\Big)^{1/u}\,.$$

In the proof  of Theorem \ref{FLuprop} we shall use an elementary convolution inequality in weighted spaces.

\begin{lemma} \label{weightineq}For $a\in \bbR$ let  $M_a$ be the multiplication operator $M_a g(r)=g(r)(1+|r|)^a$.
Let $\varsigma$ be a measurable function such that for all $N$ $$|\varsigma(r)|\le C(N)(1+|r|)^{-N}.$$ Then for all $a\in \bbR$, $1\le u\le\infty$
$$\big\| M_{-a} [\varsigma * (M_a g)]\big\|_{L^{u,s}(\mu_d)} \lc \|g\|_{L^{u,s}(\mu_d)}\,.
$$
\end{lemma}
\begin{proof} The straightforward proof is  left to the reader. Real interpolation  reduces this to the case
$u=s$ for which one can consult Lemma 2.2 in \cite{gs}.
\end{proof}

\begin{proof}[Proof of Theorem \ref{FLuprop}]
Since $\rho^{1/\beta}$ is homogeneous of degree $1$ the statement follows immediately from the special case  $\beta=1$, which we will henceforth assume. We shall write $\kappa$ for $\cK_1$, so that $\kappa=\cF^{-1}_\bbR[h]$.

\medskip

\noi{\it Proof of $\lc$ in \eqref{flqeq}.}
Let $\chi_1$ in $C^\infty$, compactly supported in $(0,\infty)$ such that $\chi_1(\rho)=1$ on the support of $h$. Since $\ka=\widehat h$  it suffices to show
$$
\big\|\cF^{-1}\big[ \chi_1\!\circ\!\rho\, \,\widehat\ka\!\circ\!\rho \big]
\big\|_{L^{u,s}(\bbR^d)} \lc
\big\|(1+|\cdot|)^{-\frac{d-1}{2}}\ka\big\|_{L^{u,s}(\mu_d)}
$$
which follows from \Be\label{suffest} \Big\|\cF^{-1} \big[
\chi_1\!\circ\!\rho \int g(r)(1+|r|)^{\frac{d-1}{2}} e^{-ir\rho} dr
\big] \Big\|_{L^{u,s}(\bbR^d)} \lc \big\|g\big\|_{L^{u,s}(\mu_d)}.
\Ee
We show the corresponding inequalities with $L^{u,s}$ replaced
by $L^1$ and by $L^2$.

For the $L^1$ inequality we use that
$$\big\|\cF^{-1} \big[ \chi_1\!\circ\!\rho\, \, e^{-ir\rho}  \big] \big\|_{L^1(\bbR^d)}
\lc (1+|r|)^{\frac{d-1}{2} }$$
which is a rescaled inequality from \cite{sesost}.  Here it is crucial that $\rho$
 is homogeneous of degree one. By Minkowski's inequality the displayed estimate yields
 \Be\label{L1suff}
 \Big\|\cF^{-1} \big[ \chi_1\!\circ\!\rho \int g(r)(1+|r|)^{\frac{d-1}{2}} e^{-ir\rho} dr \big]
\Big\|_{L^1(\bbR^d)} \lc \int |g(r)| (1+|r|)^{d-1} dr
\Ee
For the $L^2$ inequality we use polar coordinates and Plancherel's theorem in $\bbR^d$ and $\bbR$,
to  get
\begin{align}
\notag
 &\Big\|\cF^{-1} \big[ \chi_1\!\circ\!\rho \int g(r)(1+|r|)^{\frac{d-1}{2}} e^{-ir\rho} dr \big]
\Big\|_{L^2(\bbR^d)}
\\ \notag &\lc
\Big\|\chi_1\!\circ\!\rho(\cdot) \int g(r)(1+|r|)^{\frac{d-1}{2}} e^{-ir\rho(\cdot)} dr
\Big\|_{L^2(\bbR^d)}
\\ \notag &\lc
\Big\|\chi_1 \int g(r)(1+|r|)^{\frac{d-1}{2}} e^{-ir(\cdot)} dr
\Big\|_{L^2(\bbR)}
\\ \label{L2suff}
&\lc \Big(\int |g(r)|^2 (1+|r|)^{d-1}  dr \Big)^{1/2}\,.
\end{align}
The asserted inequality \eqref{suffest} follows from \eqref{L1suff} and \eqref{L2suff} by  real interpolation.

\medskip
\noi{\it Proof of $\gc$ in \eqref{flqeq}.}
We shall work with smooth $h$ which allows
to assume that the $L^{u,s}(\mu_d)$ norm of $(1+|r|)^{-\frac{d-1}{2}}\ka(r)$ is a priori finite.
With this assumption we have to prove the inequality
\Be\label{necpart}
\big\| (1+|\cdot|)^{-(d-1)/2} \ka\|_{L^{u,s}(\mu_d)}
\lc \|\cF^{-1}[ h\circ\rho]\|_{L^{u,s}(\bbR^d)} \,.
\Ee
If \eqref{necpart} holds for all smooth $h$ supported in $J$ then the general case can be derived by an approximation argument.

Pick $\xi_0\in \Sigma_\rho$ so that $|\xi|$ has a maximum at
$\xi_0$. Then the Gauss map is  injective  in a small neighborhood $U$ on the surface  $\Sigma_\rho$ and the curvature is bounded below on $U$. Let $\gamma$ be homogeneous of degree zero, $\gamma(\xi_0)\neq 0$ and supported on the closure of the cone generated by $U$. Clearly
$$\|\cF^{-1} [\gamma\,  h\!\circ\!\rho
]\|_{L^{u,s}}\lc
\|\cF^{-1} [h\!\circ\!\rho]\|_{L^{u,s}}.$$
Now use $\rho$ polar coordinates to write
\Be\label{polar}
\cF^{-1} [\gamma\,  h\!\circ\!\rho ](x)=
(2\pi)^{-d} \int_0^\infty h(\rho)\rho^{d-1}\int_{\Sigma_\rho}\gamma(\xi') e^{i\rho\inn{\xi'}{x}}
\frac{d\sigma(\xi')}{|\nabla\rho(\xi')|} \,d\rho\,.
\Ee
Let $n(\xi_0)$ the outer normal at $\xi_0$, let $\Gamma=\{x\in \bbR^d: \big|\frac{x}{|x|}-n(\xi_0)\big|
\le \eps\}$, with $\eps$ small and let, for large $R\gg1 $, $\Gamma_R=\{x\in\Gamma:|x|\ge R\}$.
By choosing $\eps$ small enough we may assume that for each $x\in \Gamma$ there is a unique $\xi=\Xi(x)\in \Sigma_\rho$, so that $\gamma(\Xi(x))\neq 0$ and so that $x$ is normal to $\Sigma_\rho$ at $\Xi(x)$. Clearly $x\mapsto\Xi(x)$ is homogeneous of degree zero on $\Gamma$.
By the method of stationary phase we have for $x\in \Gamma_R$
\Be \label {gahdec}\cF^{-1} [\gamma h(\rho(\cdot))](x)= I_0(x)+\sum_{j=1}^N II_j(x)+III(x)\Ee
where
\begin{align*}
 I_0(x)&=c\int_0^\infty h(\rho)\rho^{\frac{d-1}{2}} e^{i\rho\inn{\Xi(x)}{x}}  d\rho
 \frac{\gamma(\Xi(x))   |\nabla\rho(\Xi(x))|^{-1}}
{(\inn{\Xi(x)}{x})^{\frac{d-1}{2}}
|K(\Xi(x))|^{1/2}},
\end{align*}
$|c|=(2\pi)^{-d}$ and $K(\Xi(x))$ is the curvature of $\Sigma_\rho$ at $\Xi(x)$.
There are similar formulas for the higher order terms $II_j(x)$ with
the main term  $(\rho\inn{\Xi(x)}{x})^{-\frac{d-1}{2}}$ replaced by
$(\rho\inn{\Xi(x)}{x})^{-\frac{d-1}{2}-j}$. Finally
$$| III (x) |\lc_N \|h\|_1|x|^{-N}\,, \quad x\in \Gamma_R\,.$$

Let $h_j(\rho)=h(\rho)\rho^{\frac{d-1}{2}-j}$ and let $\kappa_j =\cF^{-1}_\bbR[h_j]$, then,
for some $C_1\ge 1$,
$$C_1^{-1}| I_0(x) | \le  \frac{|\ka( \inn{\Xi(x)}{x}) |}
{|\inn{{\Xi(x)}{x}}|^{\frac{d-1}{2}} } \,\le C_1| I_0(x)|\,,
\quad x\in \Gamma_R.$$
We also have, for some $C_0\ge 1$,
$$ C_0^{-1}|x|\le \inn{\Xi(x)}{x}\le C_0 |x|, \quad x\in \Gamma,$$  which is a consequence of Euler's homogeneity relation $\rho(\xi)= \inn{\xi}{\nabla\rho(\xi)}$ and the positivity assumption on $\rho$.
Let \begin{align*}
 E_R(\theta,\alpha)&= \{r: |r|\ge R, | I_0(r\theta) |> \alpha\}\,,
\\
E_{R}^*(\beta)&=\{r: |r|\ge C_0R: \, r^{-\frac{d-1}{2}}|\ka_0(r)|>C_1\beta\}\,.
\end{align*}
Then
\begin{align*}
\| I_0\|_{L^{u,s}(\Gamma_R)}
&\gc\Big(\int_0^\infty \big[\alpha\,(\meas\{x\in \Gamma_R:| I_0(x)|>\alpha\})^{1/u}\big]^s \frac{d\alpha}\alpha\Big)^{1/s}
\\
&\gc \Big(\int_0^\infty \Big[\alpha \Big(\int_{S^{d-1}\cap\Gamma} \int_{E_R(\theta,\alpha)} r^{d-1} dr \,d\theta\Big)^{1/u}\Big]^s  \frac{d\alpha}\alpha\Big)^{1/s}
\\
&\gc \Big(\int_0^\infty \Big[\alpha \Big(\int_{E_{R}^*(2^{d+1}\alpha)} r^{d-1} dr \Big)^{1/u}\Big]^s
\frac{d\alpha}\alpha\Big)^{1/s},
\end{align*}
which gives
\Be \label{I0below}\big\| (1+|\cdot|)^{-\frac{d-1}{2}}\ka_0\big\|_{L^{u,s}([C_0 R,\infty],\mu_d)}
\lc \| I_0\|_{L^{u,s}(\Gamma_R)} \,.
\Ee
A variant of this argument also yields  the upper bound
\Be \notag
\| I_0\|_{L^{u,s}(\Gamma_R)} \lc
\big\| (1+|\cdot|)^{-\frac{d-1}{2}}\ka_0\big\|_{L^{u,s}(\mu_d)}\,,
\Ee
and similarly, taking into account the additional decay in the terms $II_j$,
\begin{align}\notag
\| II_j\|_{L^{u,s}(\Gamma_R)} &\lc R^{-1}
\big\| (1+|\cdot|)^{-\frac{d-1}{2}}\ka_j\big\|_{L^{u,s}(\mu_d)}
\\ \label {upperboundsII}
&\lc R^{-1}
\big\| (1+|\cdot|)^{-\frac{d-1}{2}}\ka_0\big\|_{L^{u,s}(\mu_d)}\,.
\end{align}
Here,  for the second inequality, we have used Lemma \ref{weightineq} and the fact that $\ka_j=\kappa*\varsigma_j$ for some Schwartz function $\varsigma_j$.
We also have the trivial inequalities  which use the support properties of $h$
$$\| III\|_{L^{u,s}(\Gamma_R)} \lc R^{-1} \|h\|_1$$
and
$
\|h\|_1 \lc \|h\!\circ\!\rho\|_{L^{u',s}(\bbR^d)}
\lc
\|\cF^{-1}[h\!\circ\!\rho]\|_{L^{u,s}(\bbR^d)}\,,
$
by the Hausdorff-Young inequality. Thus
\Be\label{three}
\| III\|_{L^{u,s}(\Gamma_R)} \lc R^{-1} \|\cF^{-1}[h\!\circ\!\rho]\|_{L^{u,s}(\bbR^d)}.
\Ee
This argument also gives
\Be \label{I0belowtriv}\big\| (1+|\cdot|)^{-\frac{d-1}{2}}\ka_0\big\|_{L^{u,s}([0,C_0],\mu_d)}
\lc  R^d \| \cF^{-1}[h\!\circ\!\rho]\|_{L^{u,s}}
\Ee
since the left hand side is estimated by $ R^d \|\ka_0\|_\infty \lc R^d \|h\|_1$.

We now combine the estimates to show \eqref{necpart}. By \eqref{I0below} and \eqref{I0belowtriv} we get
\begin{align*}
&\big\| (1+|\cdot|)^{-\frac{d-1}2} \ka_{0} \|_{L^{u,s}(\mu_d)}
\lc \| I_0 \|_{L^{u,s} (\Gamma_R)} + R^d\|\cF^{-1}[ h\!\circ\!\rho]\|_{L^{u,s}(\bbR^d)}.
\end{align*}
Since $\|\cF^{-1}[\gamma\, h\!\circ\!\rho]\|_{L^{u,s}(\bbR^d)}
\lesssim \|\cF^{-1}[ h\!\circ\!\rho]\|_{L^{u,s}(\bbR^d)} $, by \eqref{gahdec}, \eqref{three} and \eqref{upperboundsII}
\begin{align*}
&\| I_0 \|_{L^{u,s} (\Gamma_R)}\lc (1+R^d)\|\cF^{-1}[h\!\circ\!\rho]\|_{L^{u,s}(\bbR^d)}+
 \sum_{j=1}^N\| II_j \|_{L^{u,s} (\Gamma_R)}
\\ &\lesssim  (1+R^d)\|\cF^{-1}[h\!\circ\!\rho]\|_{L^{u,s}(\bbR^d)}+  R^{-1}
\big\| (1+|\cdot|)^{-\frac{d-1}{2}}\ka_0\big\|_{L^{u,s}(\mu_d)}.
\end{align*}
Hence, combining these two estimates and choosing $R\gg1 $ sufficiently large we get
$$\big\| (1+|\cdot|)^{-\frac{d-1}2} \ka_{0} \|_{L^{u,s}(\mu_d)}
\lc \|\cF^{-1}[ h\!\circ\!\rho]\|_{L^{u,s}(\bbR^d)} \,.$$
Finally observe that $\kappa= \kappa_0*\varsigma_0$ for some Schwartz function $\varsigma_0$ and thus by Lemma \ref{weightineq}
$$\big\| (1+|\cdot|)^{-\frac{d-1}2} \ka \|_{L^{u,s}(\mu_d)}
\lc\big\| (1+|\cdot|)^{-\frac{d-1}2}  \ka_{0} \|_{L^{u,s}(\mu_d)}\,.$$
The desired estimate \eqref{necpart} follows.
\end{proof}

\subsection*{Quasiradial functions in Besov spaces}
If $1< u< 2$ we have the embedding $\hz{\alpha}{s}
\hookrightarrow L^{u,s}$
for $\alpha= d(1/u-1/2)$. This holds for $u=1$, $s=1$  by using the Cauchy-Schwarz inequality on each annulus $\fA_l$, and it trivially holds for $u=s=2$.
We use interpolation (\cite{gilbert}) together with the reiteration theorem for real interpolation to deduce that the embedding holds for $1<u<2$ and all $s>0$.

Next observe that the homogeneous Besov space $\dot B^2_{\alpha,s}$ is the image of $\hz{\alpha}{s}$ under the Fourier transform. If $h$ is supported in a compact subinterval $J$ of $(0,\infty)$ then for $\alpha>0$,
$$ \|h\!\circ\!\rho\|_{B^2_{\alpha,q}(\bbR^d)}
\lc \|h\|_{B^2_{\alpha,q}(\bbR)}  \lc \|h\|_{\dot B^2_{\alpha,q}(\bbR)};
 $$
here the implicit constants depend on $J$.
Hence we see that if $p>2$ and $h\in B^2_{d(1/2-1/p),q'}$ is supported in $J$
  then $\cF^{-1}[h(\rho)\widehat f\,]$
and the associated maximal function $\sup_{t>0}
|\cF^{-1}[h(\rho/t)\widehat f\,]$ are well defined for  $f\in
L^{p,q}$.

\subsection*{Multipliers of Bochner-Riesz type}
Now let $\{\eta_j\}$ be a family of $C^\infty$-functions supported
on $(1/4,1/2)$ so that the $C^{10 d}$ norms are uniformly bounded.
Define, for sequences $\fa=\{a_j\}_{j=2}^\infty$
\Be\label{ha}h[\fa](\tau) = \sum_{j=2}^\infty a_j 2^{-j\la}
\eta_j(2^j(1-\tau)). \Ee Then it is easy to see that for $\la>-1/2$
\Be \label{sequences}\|h[\fa]\|_{B^2_{\la+\frac 12, s}} \lc
\|\fa\|_{\ell_s}. \Ee

Now consider $h_{\la,\ga}$ in \eqref{hlaga} and let
$\chi\in C^\infty$ be supported in $(9/10, 11/10)$ and equal to $1$ near $1$.
It is well known that $[(1-\chi) h_{\la,\ga}]\circ\rho$ is the
Fourier transform
of an $L^1$ function and that the associated maximal operator is of
 weak type $(1,1)$.
The function $\chi h_{\la,\ga}$ can be written as
in \eqref{ha} for $a_j= j^{-\ga}$.  Thus
$\chi h_{\la,\ga}
\in B^2_{\la+1/2,s}$ for $\gamma s>1$.
Hence  $S^{\la,\ga}_t f$ and
the associated maximal function are well
defined on $L^{p,q}$ if $\la=d(1/2-1/p)-1/2$ and $\gamma>1-1/q$.

Finally to show lower bounds we consider the one dimensional inverse  Fourier transform
$\ka_{\la,\ga}$ of  $\chi h_{\la,\ga} $, i.e.
\begin{align*}
2\pi \ka_{\la,\ga}(r) &=  \int \chi(v) (1-v)_+^\la  (\log \frac{1}{|1-v|})^{-\ga} e^{ivr} \,dv
\\&=  e^{ir} \int \chi_0(t) t_+ ^\la  (\log |t|^{-1})^{-\ga} e^{-i tr}\,dt
\end{align*}
here $\chi_0$ is supported in $(-10^{-1},10^{-1})$.
By a standard asymptotic expansion (\cite{erd}, ch. 2.9.), we can estimate the expression  that
we get by freezing the logarithmic term at $t=r^{-1}$, namely we have for $|r|\gg 1$
$$\Big|
\int \chi_0(t) t_+ ^\la  e^{-i tr}\,dt\, (\log r )^{-\ga}
\Big| \,\gc\,|r|^{-\la-1} (\log |r|)^{-\ga}.
$$
Straightforward estimation (using integration for the parts where $t\approx r^{-1}2^m$ with $m>0$)
 shows
$$\Big|
\int \chi_0(t) t_+ ^\la
\big[ (\log |t|^{-1})^{-\ga}
- (\log r )^{-\ga} \big]
e^{-i tr}\,dt
\Big| \lc |r|^{-\la-1} (\log |r|)^{-\ga-1}.
$$
The details are left to the reader.  In view of the additional logarithmic gain in the last display we obtain for  $|r|\gg1$
$$|\ka_{\la,\ga}(r)|\gc  |r|^{-\la-1} (\log |r|)^{-\ga}\,.$$
From this it is easy to see that the condition
$$ \frac{\ka_{\la,\ga}}{(1+|r|)^{\frac{d-1}{2}}} \in L^{p',q'}((1+|r|)^{d-1}dr )$$
implies that either
 $\la> d(1/2-1/p)-1/2$ or
  $\la=d(1/2-1/p)-1/2$  and  $\ga>1/q'$.
Thus Theorem  \ref{weightineq} and Lemma \ref{TKlemma} show that  the range of $\gamma$ in Theorem  \ref{laga} is optimal.

\section{Preliminaries and  basic estimates}\label{prel}

\subsection*{An embedding result}
Recall the notation $\fA_l=\{2^l\le|x|< 2^{l+1}\}$.
\begin{lemma}\label{Loremb} Let $0<a<d$, $1<r<\infty$ and let
$p=\frac{rd}{d-a}$.
Then, for $1<q<\infty$,
$$
\Big(\sum_{l\in \mathbb Z}\Big[\int_{\fA_l}
|f(x)|^r |x|^{-a}dx\Big]^{q/r}\Big)^{1/q} \lc
\|f\|_{L^{p,q}}\,,
$$
and
$$
\sup_l \Big(\int_{\fA_l}
|f(x)|^r |x|^{-a}dx\Big)^{1/r} \lc
\|f\|_{L^{p,\infty}}\,.
$$
\end{lemma}
\begin{proof}
Let $f$ be measurable so that $|f(x)|\le \chi_E(x)$ for some set $E$ of
finite Lebesgue measure.
Let $\varpi_d$ be the surface measure of the unit sphere in $\bbR^d$ and let
$B(E)$ be the ball centered at the origin, with radius
$R(E)= (d\varpi_d^{-1}|E|)^{1/d}$; then $E$ and $B(E)$ have the same Lebesgue measure and
$$
\int_E |x|^{-a} dx \le \int_{B(E)}|x|^{-a} dx = \frac{d}{d-a}
\Big(\frac{\varpi_d}{d}\Big)^{a/d}
|E|^{1-\frac ad}\,.$$
This implies the
inequality
$$
\Big(\int |f(x)|^r |x|^{-a}dx\Big)^{1/r}\lc
\|f\|_{L^{p,1}}\,,\quad  \frac 1p= \frac 1r-\frac{a}{rd}\,.
$$
We apply this inequality
with two choices $(p_0,a_0)$ and $(p_1,a_1)$ satisfying
$a_i/r= d(1/r-1/p_i)$ for $i=0,1$ where $a_0>0$ is close to $0$ and $a_1<d$ is close to $d$.

Let $a_0<a<a_1$ and let $p=\tfrac{rd}{d-a}$. Then there is
$\vth\in (0,1)$ so that $(\tfrac ar,\tfrac 1p)=
(1-\vth) (\tfrac{a_0}r,\tfrac 1{p_0})+\vth (\tfrac{a_1}r,\tfrac 1{p_1})$.
Since $[L^{p_0,1}, L^{p_1,1}]_{\vth,q}= L^{p,q}$ and since $[L^r(|x|^{-a_0} dx),
L^r(|x|^{-a_1} dx)]_{\vth,q}= \hzr{-a/r}{q}$, by Gilbert's result
\cite{gilbert}, the assertion follows  by interpolation.
\end{proof}

\subsection*{ Fourier restriction based on  traces of
Sobolev spaces}
In what follows we let $\sigma$ be the surface  measure on the $\rho$-unit sphere
 $\Sigma_\rho$.
The following result is standard, but we include the proof for completeness.

\begin{lemma}\label{trace} For $1<b<d$,
$$\int_{\Sigma_\rho} |\widehat g(\xi)|^2 d\sigma(\xi)
\lc \int |g(x)|^2 |x|^b dx\,.
$$
\end{lemma}
\begin{proof} We split $g=g_0+g_1$ where $g_0(x)=0$ for $|x|>1$ and $g_1(x)=0$ for $|x|\le 1$.
Then
\begin{align*}
&\Big(\int_{\Sigma_\rho} |\widehat g_0(\xi)|^2 d\sigma(\xi)\Big)^{1/2}
\lc \|\widehat g_0
\|_\infty \lc \|g_0\|_1
\\&\lc \Big(\int|g_0(x)|^2 |x|^b \,dx\Big)^{1/2} \Big(\int_{|x|\le 1} |x|^{-b} dx\Big)^{1/2} \lc \Big(\int |g(x)|^2 |x|^b \,dx\Big)^{1/2}
\end{align*}
where we have used $b<d$.

Now the trace theorem says that for $b>1$
 the restriction to $\Sigma_\rho$  of functions in  the Sobolev space
 $\sL^2_{b/2}(\bbR^d)$
belongs to
$\sL^2_{(b-1)/2}(\Sigma_\rho)$, so  it is  in $L^2(\Sigma_\rho)$. We apply
the corresponding inequality to $\widehat g_1$ and combine it with Plancherel's theorem  to  see that
$$\Big(\int_{\Sigma_\rho} |\widehat g_1(\xi)|^2 d\sigma(\xi)\Big)^{1/2}
\lc \Big(\int |g_1(x)|^2 (1+|x|^2)^{b/2} dx\Big)^{1/2}.
$$ In view of the support of $g_1$ we may replace the weight
$(1+|x|^2)^{b/2}$ with $|x|^b$ and then $g_1$ with $g$. Finally, combine the estimates for
$\widehat g_0$ and $\widehat g_1$.
\end{proof}

\subsection*{Estimates for  maximal functions}
We use Lemma \ref{trace} to bound a maximal operator associated with $h\circ \rho$.
\begin{proposition}\label{basicmult} Let $1<b<d$.
Let $\rho\in C^\infty(\bbR^d\setminus \{0\}$, homogeneous of degree $\beta>0$  and positive
on $\bbR^d\setminus \{0\}$. Suppose that
$$ \Big( \int_0^\infty
|h(s)|^2s^{\frac b\beta-1}ds\Big)^{1/2} \le A.
$$
Then
$$\Big(\int \big|\sup_{1<t<2}|\cF^{-1}[h(\rho(\cdot)/t)\widehat f\,]\big|^2\frac{dx}{|x|^b}\Big)^{1/2}
\lc A \|f\|_2$$
\end{proposition}

\begin{proof} This follows from the dual inequality
\Be\label{dual}
\Big\|\int_1^2  \cF^{-1}[h\circ(\rho/t)\widehat f_t]\,dt  \Big\|_{L^2}
\lc A \Big\|\int_1^2 |f_t|dt \Big\|_{L^2(|x|^bdx)}\,.
\Ee

Using Plancherel's theorem and generalized polar coordinates
$\xi=\rho\xi'$, $\rho>0$, $\xi'\in \Sigma_\rho$, we can write the
left hand side of \eqref{dual}  as
$$\Big(\int \rho^{\frac d\beta-1}\int_{\Sigma_\rho} \Big|\int h(\rho/t)
\widehat f_t(\rho^{1/\beta}\xi')dt \Big| ^2
\frac{d\sigma(\xi')}{\beta |\nabla\rho(\xi')|} d\rho\Big)^{1/2} .
$$
Since $|\nabla \rho|$ is bounded below we can drop this term and use
Lemma \ref{trace} to get
\begin{multline*}
\Big\|\int_1^2  \cF^{-1}[h\circ(\rho/t)\widehat f_t]dt  \Big\|_{L^2}\\
\lc
\Big(\int \rho^{\frac d\beta -1}
\int |x|^b \Big|\int h(\rho/t)\rho^{-d/\beta}
f_t(\rho^{-1/\beta}x)dt \Big| ^2  dx\,
d\rho \Big)^{1/2}.
\end{multline*}
We change variables
 and the last expression becomes
\begin{align*}
&\Big(\int |x|^b \int \rho^{\frac b\beta-1}
\Big|\int h(\rho/t) f_t(x)dt \Big| ^2  dx\,d\rho\Big)^{1/2}\\
&\le
\Big(\int |x|^b \Big[\int_{1}^2 \Big( \int
| h(\rho/t)|^2 \rho^{\frac b\beta-1} d\rho\Big)^{1/2} |  f_t(x)|\,dt
\Big]^2dx\Big)^{1/2}\\
&\le \Big( \int \rho^{\frac b\beta-1}
| h(\rho)|^2d\rho\Big)^{1/2}
\Big(\int
\Big[\int_{1}^2 t^{\frac{b}{\beta}}
|  f_t(x)|\,dt \Big]^2|x|^b dx\Big)^{1/2}.
\end{align*}
The assertion follows.
\end{proof}

In the proof of Theorem \ref{sq} we shall use
\begin{proposition}  \label{Atauprop}
Let $1<b<d$ and $\rho$ as in Proposition \ref{basicmult}. Let $\eta\in C^\infty$ be supported in $(\frac 18,8)$.
For any real $\tau$ define $\cA_\tau$ by
\Be\label{Atau}\widehat {\cA_\tau f}(\xi) =
\eta(\rho(\xi))
e^{-i\rho(\xi) \tau}
 \widehat f(\xi)\,.
\Ee
Then
$$\int_{-\infty}^\infty\|\cA_\tau f\|_{L^2(|x|^{-b})}^2  d\tau \lc \|f\|_{2}^2\,.$$
\end{proposition}
\begin{proof}
The inequality is equivalent with
\Be\label{dualsqf}
\Big\|\int \cA_\tau [g(\tau, \cdot)] d\tau \Big\|_2 \lc
\Big(\int\|g(\tau,\cdot)\|_{L^2(|x|^b)}^2 d\tau\Big)^{1/2}\,.
\Ee
Using Plancherel's theorem and generalized polar coordinates we see that
the square of the left hand side is bounded by
\begin{align*}
&\int|\eta(\rho)|^2 \int_{\Sigma_\rho}\Big|
\int e^{i\tau \rho} \widehat g(\tau, \rho^{1/\beta}\xi') d\tau \Big|^2 \frac{d\sigma(\xi')}
{|\nabla\rho(\xi')|} \rho^{\frac d\beta -1}d\rho
\\
&\lc \int|\eta(\rho)|^2 \int\Big|
\int e^{i\tau \rho} \rho^{-d/\beta}g(\tau, \rho^{-1/\beta}x)\, d\tau \Big|^2
|x|^{b} dx  \, \rho^{\frac d\beta -1} d\rho
\end{align*}
where we have used Lemma \ref{trace}.
By a change of variables the last expression is
\begin{align*}
&\int|\eta(\rho)|^2\rho^{\frac b\beta-1}
   \int\Big|
\int e^{i\tau \rho} g(\tau, x) d\tau \Big|^2
|x|^{b} dx \,d\rho
\\
&\lc  \Big\| \Big(\int \Big|
\int e^{i\tau \rho} g(\tau, \cdot) d\tau \Big|^2 d\rho\Big)^{1/2}
\Big\|_{L^2(|x|^{b})dx}^2
\\
&\lc  \Big\| \Big(\int |
g(\tau, \cdot)|^2 d\tau \Big)^{1/2}
\Big\|_{L^2(|x|^{b})dx}^2
\end{align*}
where the last inequality holds
by Plancherel's theorem in the first variable.
Now  \eqref{dualsqf} follows.
\end{proof}

\subsection*{Littlewood-Paley inequalities}
Let $\fA$ be a compact subset of $\bbR^d\setminus \{0\}$ and let $\{\zeta_k\}$ be a bounded family in $C^\infty$, so that all $\zeta_k$ are supported in $\fA$.
Define operators $P_k$ by $\widehat {P_k f}(\xi)=\zeta_k(2^{-k} \xi)\widehat f$.

\begin{lemma}\label{LP} Let $-d/2<\ga<d/2$. Then for all
$f\in\hz{\ga}{q}$,
$$ \Big\|\Big(\sum_k|P_k f|^2\Big)^{1/2} \Big\|_{\hz{\ga}{q}} \le C(\fA,\ga)
\|f\|_{\hz{\ga}{q}}$$
and
for  all $\ell^2$ valued functions
 $F=\{f_k\}_{k\in \bbZ} \in \hz{\ga}{q}(\ell^2)$
$$ \Big\|\sum_kP_k f_k \Big\|_{\hz{\ga}{q}} \le C(\fA,\ga)
\Big\|\Big(\sum_k|f_k|^2\Big)^{1/2}\Big\|_{\hz{\ga}{q}}\,.
$$
\end{lemma}
\begin{proof} Since $|x|^{-2\gamma}$ is an $A_2$ weight for $-d/2<\gamma<d/2$ the Littlewood-Paley inequalities hold for the weighted $L^2(|x|^{2\gamma} dx)$ spaces, i.e. for $\hz{\ga}{2}$. The case for general $q$ follows by real
interpolation.
\end{proof}

\subsection*{Reduction to the case of homogeneity $\beta=1$}\label{betaeqone}
In the proof of our main theorems we may assume $\beta=1$. This follows
from  writing $h(\rho(\xi))= h_\beta(\rho(\xi)^{1/\beta})$ with $h_\beta(s)=h(s^\beta)$ and the following lemma.

\begin{lemma} Let $\alpha>0$, $\beta>0$. Then for all
 $h\in B^2_{\alpha, q}(\bbR)$ supported in a compact
interval $J\subset(0,\infty)$
$$\|h((\cdot)^\beta)\|_{B^2_{\alpha,q}}\le C_{\beta,J }
\|h\|_{B^2_{\alpha,q}}\,.
$$
\end{lemma}
\begin{proof} Let $J=[a,b]$ and $J_\beta=[a^{1/\beta}, b^{1/\beta}]$.
Let $\chi\in C^\infty_0$ be compactly supported in $(0,\infty)$ and equal to $1$ on $J_\beta$. Then we have
$\|\chi(\cdot)h((\cdot)^\beta)\|_{\sL_k^2}\le C(\chi,\beta)
\|h\|_{B^2_{\alpha,q}}\,
$ for $k=0,1,2,\dots$, by straightforward computation. The corresponding inequality for Besov spaces $B^2_{\alpha,q}$ for all  $\alpha>0$
follows by real interpolation.
 \end{proof}

\section{Proof of Theorem \ref{herzest}}\label{herzestsect}
In what follows we shall assume  $\beta=1$, as we may by the last subsection in \S\ref{prel}.
\subsection*{The basic decompositions}\label{basicdec}
Let $h\in B^2_{\alpha,q'}(\bbR)$ be supported in $(1/2,2)$ and let $\eta $ be  a $C^\infty$ function,  supported on $(1/8,8)$ such that
$\eta(s)=1$ on $(1/4,4)$.
Define $L_k$ by
$\widehat {L_k f}(\xi)=\eta(2^{-k}\rho(\xi))\widehat f(\xi)$.
Then
\begin{align}\notag\sup_{s>0} \big|\cF^{-1}[h(\tfrac \rho s) \widehat f\,]\big|
&=\sup_{k}\sup_{1\le t\le 2} \big|\cF^{-1}[h(\tfrac{\rho}{2^kt})
\widehat {L_kf}]\big|\,\\
\label{maxsqfct}
&=\Big(\sum_{k}\sup_{1\le t\le 2}
\big|\cF^{-1}[h(\tfrac{\rho}{2^kt})
\widehat {L_kf}]\big|^2 \Big)^{1/2}\,.
\end{align}

Let $\zeta_0\in C^\infty(\bbR)$ be supported in $(-1,1)$ such that
$\zeta_0(r)=1$ for $|r|\le 1/2$ and let, for $j\ge 1$,
$\zeta_j(r)=\zeta_0(2^{-j}r)-\zeta_0(2^{-j+1}r)$.
Define
\Be\label{Vjdef}
V_jh (\rho)=(2\pi)^{-1}\int\zeta_j(r) \widehat h(r) e^{ir\rho}\, dr
\Ee
so that $\widehat {V_j h}$ is supported in $I_j=[2^{j-2}, 2^j]\cup [-2^j,-2^{j-2}]$ for $j=1,2,\dots$, and  $\sum_{j=0}^\infty V_j h=h$. Define
 \Be\label{Tjtaudef}T^{j,k}_s[h, f]=\cF^{-1}[ \eta(\rho(2^{-k}\cdot))V_jh(2^{-k}
\tfrac{\rho(\cdot)}s)\widehat f\,]\,.\Ee

We first note that for large $j$
 the Fourier transform of $\eta(\rho(\cdot))V_jh
(\rho(\cdot))$ is concentrated on  an annulus of width $\approx
2^j$. By assumption, there is $c_0\in \bbN$ so that \Be\label{c0def}
2^{-c_0+2}\le |\nabla\rho(\xi)|\le 2^{c_0-2}, \text { for } 1/8\le
\rho(\xi)\le 8. \Ee
\begin{lemma} \label{errorlemma}
Let $K_{j,s}$ be defined by $\widehat{K}_{j,s}(\xi)= \eta(\rho(\xi)) V_j h(\rho(\xi)/s)$.
Then  for $j=0,1,2,\dots$, and $1/2\le s\le 2$,
\Be\label{Kjerrors}\begin{aligned}
|K_{j,s}(x)|&\le C_N \|h\|_\infty 2^{-jN}|x|^{-N}, \text{ for }
|x|\ge 2^{j+c_0}\,,
\\
|K_{j,s}(x)|&\le C_N \|h\|_\infty 2^{-jN}, \text{ for }
|x|\le 2^{-j-c_0}\,.
\end{aligned}
\Ee
\end{lemma}
\begin{proof}
First observe that $\|\widehat{ V_j h}\|_1\lc 2^j\|h\|_\infty$. We write
$$K_{j,s}(x)= \frac{1}{(2\pi)^{d+1}}
\int \widehat {V_jh}(\tau) \int \eta(\rho(\xi))
e^{i(\tau\rho(\xi)+\inn x\xi)}\, d\xi\, d\tau
$$
and observe that, for $\xi \in \supp (\eta\circ\rho)$,
$$|\nabla_\xi (\tau \rho(\xi)+\inn x\xi)|\ge \max\{
|x|-2^{j+c_0-1}, 2^{j-c_0}-|x|\}$$
Multiple integration by parts in $\xi$ yields the asserted estimate.
\end{proof}

Let $\chi_l$ be the indicator function of the annulus $\fA_l$.
The above lemma suggests to
estimate  the maximal square-function \eqref{maxsqfct} by
$$\Big(\sum_k \cM^k_1 [L_k f,h]^2\Big)^{1/2}
+\Big(\sum_k \cM^k_2 [L_kf,h]^2\Big)^{1/2}
$$
where
\Be \label{M12def}
\begin{aligned}
\cM_1^k[ f,h] &= \sup_{1\le t\le 2}\Big|\sum_{j=0}^\infty T^{j,k}_{ t}\big[h,\sum_{n\ge -c_0} \chi\ci{-k+j-n} f\big]\Big|\,,
\\
\cM_2^k [f,h]&= \sup_{1\le t\le 2}\Big|\sum_{j=0}^\infty T^{j,k}_{ t}\big[h,\sum_{l>j+c_0} \chi\ci{-k+l} f\big]\Big|\,.
\end{aligned}
\Ee
For the second  term we can reduce to straightforward $L^2$ estimates
for which the weight plays little role.
The first term requires the condition $h\in B^2_{\alpha,s}$  but
 we get a  bound even for
 $\sum_k \cM_1^k[f_k,h]$ in place of the square-function.

\begin{proposition} \label{M1prop}  Let $1/2<\alpha< d/2$, $s=\frac{q}{q-1}$
and  let $h\in B^2_{\alpha,s}$ be supported in $(\frac12,2)$.

(i) If $2\le q\le\infty$ then
$$
\Big\|\sum_k \cM_1^k[f_k,h]
 \Big\|_{\hz{-\a}{q}} \lc \|h\|_{B^2_{\a,s}}
\Big\|\Big(\sum_k |f_k|^q\Big)^{1/q}\Big\|_{\hz{-\a}{q}}\,.
$$

(ii) If $1\le q\le 2$ then
$$
\Big\|\sum_k \cM_1^k[f_k,h] \Big\|_{\hz{-\a}{2}} \lc \|h\|_{B^2_{\a,s}}
\Big\|\Big(\sum_k|f_k|^2\Big)^{1/2}\Big\|_{\hz{-\a}{q}}\,.
$$
\end{proposition}

\begin{proposition}  \label{M2prop} Let $1/2<\alpha< d/2$, $\gamma>1/2$
 and  let $h\in B^2_{\gamma,1}$ be supported in $(\tfrac 12,2)$. Then
$$
\Big
\|\Big(\sum_k\cM_2^k[f_k,h]^2\Big)^{1/2}\Big\|_{\hz{-\a}{q}}\,
\lc \|h\|_{B^2_{\gamma,1}}
\Big\|\Big(\sum_k|f_k|^2\Big)^{1/2}\Big\|_{\hz{-\a}{q}}\,.
$$
\end{proposition}

We prove these propositions in the following two subsections.

\subsection*{Estimates for $\cM_1$} We localize both on the function and the operator sides. The basic estimate is
\begin{lemma} \label{fixed-kmn} Let $1<b<d$ and $1/2<\a<d/2$.
Let $V_j$ be as in \eqref{Vjdef}
and
\Be \label{Lambdajdef}
 \Lambda_b^j(h)=
\Big(\int |V_j h(\rho)|^2\rho^{b-1} d\rho\Big)^{1/2}\,.
\Ee
For $k\in \bbZ$, $m\in \bbZ$, and $n\ge -c_0$,
\begin{multline*}\Big\|\chi_{m-k} \sup_{1\le t\le 2}\Big|
\sum_{j=0}^\infty T^{j,k}_{t}[h, g\chi\ci{-k+j-n}]\Big| \,\Big\|_{L^2(|x|^{-2\a}dx)}
\\
\lc 2^{-n\a} 2^{m(\frac b2-\a)} \sum_{j\ge 0}  2^{j\a}\Lambda_b^j(h)
 \|g \chi_{j-n-k}\|_{L^2(|x|^{-2\a}dx)} \,.
\end{multline*}
\end{lemma}
\begin{proof}
On the support of $\chi_{m-k} $ we have
$|x|^{-2\alpha} \approx 2^{(k-m)(2\a-b)}
|x|^{-b}$ and by Minkowski's inequality
\begin{align}\notag
&\Big\|\chi_{m-k} \sup_{1\le t\le 2}\Big|
\sum_{j=0}^\infty T^{j,k}_{ t}[h, g\chi\ci{-k+j-n}]\Big| \,\Big\|_{L^2(|x|^{-2\a}dx)}
\\
\label{bpower}
&\lc
2^{(k-m)\frac{2\a-b}{2}}
\sum_{j=0}^\infty \big\|\sup_{1\le t\le 2}\big| T^{j,k}_{t}[h, g\chi\ci{-k+j-n}]
\big|\,\big\|_{L^2(|x|^{-b}dx)}\,.
\end{align}
Note that
$T^{j,k}_{ t}[h,g](x)= T^{j,0}_{t}[h,g(2^{-k}\cdot)](2^kx)$ and hence
\begin{align*}
&\big\|\sup_{1\le t\le 2}\big| T^{j,k}_{ t}[h, g\chi\ci{-k+j-n}]
\big|\,\big\|_{L^2(|x|^{-b}dx)}\,
\\
&=2^{- k \frac{d-b}2}
\big\|\sup_{1\le t\le 2}\big| T^{j,0}_{t}[h, g(2^{-k\cdot})
\chi\ci{j-n}]
\big|\,\big\|_{L^2(|x|^{-b}dx)}\,
\\&\lc 2^{- k \frac{d-b}2} \Lambda_b^j(h) \big\|g(2^{-k}\cdot) \chi_{j-n}\|_2
\end{align*}
where we have applied Proposition \ref{basicmult}. The last displayed quantity   is equal to
\begin{align*}
&2^{ kb/2} \Lambda_b^j(h)
\big\|g\chi_{-k+j-n}\|_2
\lc
2^{ k\frac{b-2\a}2}2^{(j-n)\a} \Lambda_b^j(h)
\big\|g\chi_{-k+j-n}\|_{L^2(|y|^{-2\a}dy)}
\end{align*}
and the asserted inequality follows if we use this in
\eqref{bpower}.
\end{proof}

\begin{lemma} \label{Lambda-Besov} Let $h\in L^2(\bbR)$ be  supported in $(\tfrac 12, 2)$
and $\Lambda_b^j(h)$ be as in \eqref{Lambdajdef}. Then for $b>0$, $\a>0$
$$\Big(\sum_{j=0}^\infty [2^{j\a}\Lambda_b^j (h)]^s\Big)^{1/s}\lc \|h\|_{B^2_{\a,s}}.$$
\end{lemma}
\begin{proof}
Let $$\widetilde \Lambda^j_b(h)\, = \,\int_{1/4}^4|V_j h(\rho)|^2\rho^{b-1}d\rho.$$ Then the corresponding estimate with
$\Lambda^j_b$ replaced by $\widetilde \Lambda^j_b$ is obvious from the definition of Besov spaces. To estimate the corresponding contribution over $\bbR\setminus[1/4,4]$ write
$$ 2\pi V_jh(\rho)= \int h(u)\int\chi_j(r) e^{i (\rho-u)r} dr \,du.$$
By integration by parts, the inner integral is  $\le C_N 2^{-jN}|\rho-u|^{-N-1}$
and since
$h$ is supported in $[\tfrac 12,2]$ we see that
$|V_jh(\rho)|\lc \|h\|_12^{-jN}$ for $\rho\le 1/4$ and
$|V_jh(\rho)|\lc \|h\|_1 2^{-jN}\rho^{-N}$ for $|\rho|\ge 4$.
A straightforward estimate gives the assertion.
\end{proof}

\begin{proof}[Proof of Proposition \ref{M1prop} for $q\ge 2$]
Let
\Be\label{Mmndef}
\cM_{m,n}^k f
=\chi_{m-k}
\sup_{1\le t\le 2}\Big|\sum_j T^{j,k}_{t}\big[h,\chi\ci{-k+j-n} f\big]\Big|\,.
\Ee
Then, by Minkowski's inequality,
\Be\label{Mink-mn}
\Big\|\sum_k\cM_1^k[ f,h]\Big\|_{\hz{-\a}{q}}
\le \sum_{m=-\infty}^\infty \sum_{n=-c_0}^\infty
\Big\|\sum_k\cM_{m,n}^k[ f,h]\Big\|_{\hz{-\a}{q}}.
\Ee
We  use the estimate
\Be\label{obviousbq}
\Big\|\sum_k \chi_{m-k} g_k\Big\|_{\hz{-\a}{q}}
\lc \Big(\sum_k \big\|\chi_{m-k}g_k\big\|^q_{L^2(|x|^{-2\a})}\Big)^{1/q}
\Ee
which holds  for all $q\ge 1$ and all $\a\in \bbR$ and is
 obvious from the definition of the ${\hz{-\a}{q}}$ spaces.
 We now choose $b_1$, $b_2$ such that $1<b_1<2\alpha$ and $2\alpha<b_2<d$.
From
\eqref {Mink-mn}, \eqref{obviousbq}
 and Lemma \ref{fixed-kmn}  we get
\begin{multline*}\Big\|\sum_k\cM_1^k[ f_k,h]\Big\|_{{\hz{-\a}{q}}}\,\lc\, \sum_{n\ge -c_0}2^{-n\alpha} \,\times \\\Big[
\sum_{m\ge 0} 2^{-m \frac{2\alpha-b_1}2} \Big(\sum_k |\cE_{b_1,n,k}|^q\Big)^{1/q}
+
\sum_{m<0} 2^{m\frac{b_2-2\alpha}{2}} \Big(\sum_k|\cE_{b_2,n,k}|^q\Big)^{1/q}\Big]
\end{multline*}
where
\Be\label{Ebn}\cE_{b,n,k} \,=\,
\sum_j 2^{j\alpha}\Lambda_b^jh \big\|f_k \chi_{j-n-k}\big\|_{L^2(|y|^{-2\alpha}dy)}
\,.
\Ee
Now the proof is concluded by  Lemma \ref{Lambda-Besov} and
estimating (with $b=b_1$ or $b=b_2$)
\begin{align*}
&\Big(\sum_k|\cE_{b,n,k}|^q\Big)^{1/q}\\&\lc
\Big(\sum_k
\Big(\sum_j [2^{j\alpha}\Lambda_b^j(h)]^s\Big)^{q/s}
\sum_j
 \big\|f_k \chi_{j-n-k}\big\|_{L^2(|y|^{-2\alpha}dy)}^q\Big)^{1/q}
\\&=\Big(\sum_j [2^{j\alpha}\Lambda_b^j(h)]^s\Big)^{1/s}
\Big(\sum_l\sum_k
\big\|f_k \chi_{l}\big\|_{L^2(|y|^{-2\alpha})}^q\Big)^{1/q}
\\&\lc \|h\|_{B^2_{\a,s}}
\Big\|\Big(\sum_k |f_k|^q\Big)^{1/q}\Big\|_{\hz{-\a}{q}}\,. \qedhere
\end{align*}
\end{proof}

\begin{proof}[Proof of Proposition \ref{M1prop} for $q\le 2$]
Again we work with \eqref{Mmndef}, use \eqref{Mink-mn}, \eqref{obviousbq}
 for $q=2$ and Lemma \ref{fixed-kmn} to get
 \begin{multline*}\Big\|\sum_k\cM_1^k[ f_k,h]\Big\|_{{\hz{-\a}{2}}}\,\lc\, \sum_{n\ge -c_0}2^{-n\alpha} \,\times \\\Big[
\sum_{m\ge 0} 2^{-m \frac{2\alpha-b_1}2} \Big(\sum_k |\cE_{b_1,n,k}|^2\Big)^{1/2}
+
\sum_{m<0} 2^{m\frac{b_2-2\alpha}{2}} \Big(\sum_k|\cE_{b_2,n,k}|^2\Big)^{1/2}\Big],
\end{multline*}
with $\cE_{b,n,k}$ as in \eqref{Ebn}, and
 $b_1\in (1,2\alpha)$ and $b_2\in(2\alpha,d)$.
Again we apply H\"older's inequality and get, for fixed $m, n$,
\begin{align*}
&\Big(\sum_k|\cE_{b,n,k}|^2\Big)^{1/2}
\\
&\le\Big(\sum_k
\Big(\sum_j [2^{j\alpha}\Lambda_b^j(h)]^s\Big)^{2/s}\Big(\sum_j \big\|f_k \chi_{j-n-k}\big\|_{L^2(|y|^{-2\alpha}dy)}^q\Big)^{2/q}\Big)^{1/2}
\\
&\le \Big(\sum_j [2^{j\alpha}\Lambda_b^j(h)]^s\Big)^{1/s}
\Big(\sum_k\Big[\sum_l \big\|f_k \chi_{l}\big\|_{L^2(|y|^{-2\alpha}dy)}^q\Big]^{2/q}
\Big)^{1/2}
\end{align*} and since now $q\le 2$ we may use Minkowski's inequality
and estimate
this by
\begin{align*}
&\|h\|_{B^2_{\a,s}}   \Big(\sum_l \Big\|\chi_l\Big(\sum_k|f_k|^2\Big)^{1/2}
\Big\|_{L^2(|y|^{-2\alpha}dy)}^q\Big)^{1/q}
\\
&\lc
\|h\|_{B^2_{\a,s}}    \Big\|\Big(\sum_k|f_k|^2\Big)^{1/2}
\Big\|_{\hz{-\a}{q}}\,. \qedhere
\end{align*}
\end{proof}
\subsection*{Estimates for  $\cM_2$}\label{m2sect}
For every $l$ let
$\widetilde\chi\ci{l}= \sum_{i=-2-c_0}^{2+c_0} \chi_l$.
We estimate
$$\cM_2^k[f,h] \le \sum_{j\ge 1} \Big(\widetilde \cM_{2,j}^k[f,h]
+\cN_j^k[f,h]\Big)\,,$$ where \Be \label{M2defs}
\begin{aligned}
\widetilde \cM_{2,j}^k[ f,h] &= \sum_{l>j+c_0} \sup_{1\le t\le 2}\Big| \widetilde \chi\ci{-k+l}
T^j_{2^k t}\big[h, \chi\ci{-k+l} f\big]\Big|\,,
\\
\cN_j^k [f,h]&= \sum_{l>j+c_0}\sup_{1\le t\le 2}\Big|(1-\widetilde \chi\ci{-k+l})\sum_j T^j_{2^k t}\big[h,
 \chi\ci{-k+l} f\big]\Big|\,.
\end{aligned}
\Ee
For the estimation of $(\sum_k\widetilde \cM_{2,j}^k[f,h] ^2)^{1/2}$
we use a standard imbedding estimate.

\begin{lemma}\label{Sobemb}
For $\ga>1/2$
$$\big\|\sup_{1\le t\le 2}\big|T_{2^kt}^j [h,f]\big|\big\|_2
\le 2^{j(\frac 12 -\gamma)}  \| h\|_{B^2_{\ga,1} } \|f\|_{L^2(\bbR^d)}.$$
\end{lemma}
\begin{proof}
By scaling it suffices to prove this for $k=0$.
For $g$ supported in $(1/2,2)$  and $I=[1/2,2]$,
$$\|\sup_{t\in I}|\cF^{-1}[g(\tfrac{\rho(\cdot)}{t})\widehat f\,]|\big\|_2 \le
\big[ \|g\|_2 + \|g\|_2^{1/2}\|g'\|_2^{1/2}\big] \|f\|_2\,,
$$
as one can see by applying the fundamental theorem of calculus to $g^2$. Since
  $\|V_j h\|_2 + 2^{-j}\|(V_j h)'\|_2 \le
2^{-j\gamma} \|h\|_{B^2_{\gamma,1}}$,
the assertion follows.
\end{proof}

\begin{proof}[Proof of Proposition \ref{M2prop}]
From Lemma \ref{Sobemb} we get, considering the supports of the
functions $\chi\ci{-k+l}$ and  $\widetilde \chi\ci{-k+l}$,
\begin{align*}
\\ &\Big\| \Big(\sum_k \widetilde \cM_{2,j}^k[ f_k,h]^2\Big)^{1/2}
\Big\|_{L^2(|x|^{-2\a} dx)}
\\
&\lc \Big(\sum_{l> c_0+j}\sum_k2^{2\alpha(k-l)}
\big\|\sup_{1\le t\le 2}\big|
T^{j,k}_{t}\big[h, \chi\ci{-k+l} f_k ]\big|\big\|_2^2 \Big)^{1/2}
\\
&\lc
2^{j(\frac 12-\ga)}\|h\|_{B^2_{\ga,1}}
\Big(\sum_{l}\sum_k2^{2\alpha(k-l)}
\big\|f_k\chi\ci{-k+l}\big\|_2^2\Big)^{1/2}
\\&\lc 2^{j(\frac 12-\ga)}\|h\|_{B^2_{\ga,1}}
\Big\|\Big(\sum_k|f_k|^2\Big)^{1/2}\Big\|_{L^2(|x|^{-2\a}dx)}\,.
\end{align*}

From Lemma \ref{errorlemma} we get
$$\cN_{j}^k[ f,h](x) \lc 2^{-jN} \|h\|_\infty \int
\frac{2^{kd}}{(1+2^k|x-y|)^N}
|f(y)|\, dy\,,
$$
and since $|x|^{-2\a}$ is an $A_2$ weight for $|\a|<d/2$ the estimate
$$\Big\| \Big(\sum_k \cN_{j}^k[ f_k,h]^2\Big)^{1/2}
\Big\|_{L^2(|x|^{-2\a}dx)}
\lc 2^{-jN}\|h\|_\infty
\Big\|\Big(\sum_k|f_k|^2\Big)^{1/2}\Big\|_{L^2(|x|^{-2\a}dx)}
$$
follows from standard results (see ch. IV in  \cite{g-r}).
Proposition
\ref{M2prop} follows by  combining the estimates for  $\{\widetilde\cM^k_{2,j}\}$ and $\{\cN^k_j\}$.
\end{proof}

\begin{proof}[Proof of Theorem \ref{herzest}] By \eqref{maxsqfct}, the decomposition described in \S\ref{basicdec} and
Propositions \ref{M1prop} and \ref{M2prop} we have the estimates
$$\Big\|\sup_{t>0}|\cF^{-1}[h(\tfrac{\rho(\cdot)}{t})\widehat f\,]|
\Big\|_{\hz{-\alpha}{\sigma}} \lc
\Big\|\Big(\sum_k|P_k f|^2\Big)^{1/2}\Big\|_{\hz{-\alpha}{q}}, \text{ $\sigma=\max\{q,2\}$,}
$$ where
the $P_k$ are suitable Littlewood-Paley operators as used in Lemma \ref{LP}, satisfying $L_kP_k=L_k$. We conclude by an application of
the first inequality in Lemma \ref{LP}.
\end{proof}

\section{Conclusion of other proofs}\label{concl}

\begin{proof}[Proof of Theorem \ref{multch}] $(i)$ and $(ii)$ are equivalent by duality.

Assume $(ii)$ holds. Define  $T_t $ by $\widehat{ T_tf}(\xi)=
m(t\rho(\xi))
\widehat f(\xi)$.
 Then $T_t$ is bounded on $L^2(|x|^{2\alpha})=\hz{\alpha}{2}$, with operator norm independent of $t$,
 by the homogeneity of $\rho$ and scale invariance. Test $T_t$ on Schwartz functions of the form $\varphi\circ \rho$
 to derive the condition $(iii)$.

Finally if $(iii)$ holds, let $\chi\in C^\infty(\bbR)$ be supported
in $(1/2,2)$ so that $\sum_{k\in\bbZ}[\chi(2^{-k}s)]^3=1$. Define
$P_k$ and  $T_k$ by $\widehat {P_k
f}(\xi)=\chi(2^{-k}\rho(\xi))\widehat f(\xi)$ and $\widehat {T_k
f}(\xi)=\chi(2^{-k}\rho(\xi))m(\rho(\xi))\widehat f(\xi)$.
 It follows from Propositions
\ref{M1prop} and \ref{M2prop}  that
$P_kT_kP_k$ is bounded on $\hz{-\alpha}{2}$ with operator norm $\lc
\|\chi m(2^k\cdot)\|_{\sL^2_\alpha}$. Hence  by interchanging sums and integrals
$$\Big\|\Big(\sum_k|P_kT_kP_k f|^2\Big)^{1/2}\Big\|_{\hz{-\alpha}{2}}
\lc \sup_k \|\chi m(2^k\cdot)\|_{\sL^2_\alpha}
\Big\|\Big(\sum_k|P_kf|^2\Big)^{1/2}\Big\|_{\hz{-\alpha}{2}}. $$
A standard estimation gives
$\sup_k \|\chi m(2^k\cdot)\|_{\sL^2_\alpha}
\le \sup_t \|\varphi m(t\cdot)\|_{\sL^2_\alpha} $.
We finish by applying both Littlewood-Paley inequalities in Lemma \ref{LP}.
\end{proof}

\begin{proof}[Proof of Corollary  \ref{herzcor}]
Given Theorem  \ref{herzest} this is a standard argument.
Define $S_t$ by
$\widehat {S_tf}(\xi)=\chi(\rho(\xi)/t)h(\rho(\xi)/t)$ with
 $h\in B^2_{\alpha, q'}$.
Let $f\in L^{p,q}$. We need to show
$\lim_{t\to\infty}S_t f(x)=0$ a.e.  on $U_R=\{x:R^{-1}\le |x|\le R\}$, for every $R>2$. For this it suffices to show that for every $r>0$, $\eps>0$ we have
\Be \meas\{x\in U_R: \limsup_{t\to \infty}|S_t f(x)|>r\}<\eps\,.
\label{limsup}
\Ee
It is easy to see that for any Schwartz function $g$
we have
$\lim_{t\to\infty}S_t g(x)=0$  uniformly on compact sets.
Let $\sigma=\max\{q,2\}$.
Given any Schwartz function $g$ the left hand side of \eqref{limsup} can now be estimated by
\begin{align*}
&\meas\{x\in U_R: \limsup_{t\to \infty}|S_t [f-g](x)|>r\}
\\ &\lc r^{-2} \| \sup_{t>0} |S_t [f-g]|\|_{L^2(U_R)}
\\ &\lc r^{-2} R^{d(1/2-1/p)}(\log R)^{1/2-1/\sigma}\| \sup_{t>0}
|S_t [f-g]|\|_{\hz{d(\frac12-\frac 1p)} {q}}\,.
\end{align*}
If we combine Theorem \ref{herzest} with Lemma  \ref{Loremb}  we get
$$\| \sup_{t>0}
|S_t [f-g]|\|_{\hz{d(\frac12-\frac 1p)} {q}}\,\lc \|h_{\la,\ga}\|_{B^2_{d(1/2-1/q),q'}}
\|f-g\|_{L^{p,q}}\,.
$$
The implicit constants in these estimates are independent of $g$.
Since $\cS$ is dense in $L^{p,q}$ we may find $g\in \cS$ to make the last
expression as small as we wish, and \eqref{limsup} follows. \end{proof}

\begin{proof}[Proof of Theorem \ref{laga}]\
Let $\chi\in C^\infty_0$ supported in $(1/2,2)$ and equal to $1$ in a neighborhood of $1$. We have $h_{\la,\ga}(\rho/t)\widehat f(\xi)=
\widehat T_t f(\xi)+ \widehat S_t f(\xi)$ where
 $T_t$ is defined by   by
$\widehat {T_tf}(\xi)=(1-\chi(\rho(\xi)/t)) h_{\la,\ga}(\rho(\xi)/t)
\widehat f(\xi)$. The maximal function $\sup_t|T_tf|$ is controlled by the Hardy-Littlewood maximal function and it is standard that $\lim T_t f(x)=f(x)$ almost everywhere for every $f\in L^r$, $r\ge 1$.
Next
 $h_{\la,\ga} \in B^2_{\alpha, q'}$
 for
$\alpha=d(1/2-1/p)$ and $\gamma>1-1/q$, moreover
 $h_{\la,0} \in B^2_{\alpha, \infty}$.
By Corollary \ref{herzcor},  $S_t f(x)\to 0$ a.e. for $f\in
L^{p,q}$. This proves the theorem.
\end{proof}


\section{The square-function estimate}\label{sqsect}
We now give a proof of Theorem \ref{sq}.
Let $\rho \in C^\infty_0(\bbR^d)$ be homogeneous of degree $1$, with $\rho(\xi)>0$ for $\xi\neq 0$.
We aim to prove, for $1/2<\alpha<d/2$,
\Be\label{Galphabeta}
\Big\|\Big(\int_0^\infty \big|\cF^{-1}
\big[\frac{\rho^\beta}{t^\beta}
(1-\frac{\rho^\beta}{t^\beta})_+^{\alpha-1} \widehat f
\,\big]\big|^2 \frac{dt}{t}\Big)^{1/2}
\Big\|_{L^2(|x|^{-2\alpha}dx)} \lc \|f\|_{L^2(|x|^{-2\alpha}dx)}\,.
\Ee
A change of variables shows that \eqref{Galphabeta} implies
Theorem \ref{sq}. Let $\eta_0\in C^\infty_0$ be supported on  $(-5/6,5/6)$ and
 equal to $1$ on $(-3/4,3/4)$. By standard weighted norm estimates for
vector-valued singular integrals we have
$$
\Big\|\Big(\int_0^\infty \big|\cF^{-1} \big[\eta_0(\frac\rho t)
\frac{\rho^\beta}{t^\beta}
(1-\frac{\rho^\beta}{t^\beta})_+^{\alpha-1} \widehat f
\,\big]\big|^2 \frac{dt}{t}\Big)^{1/2} \Big\|_{L^2(|x|^{b}dx)} \lc
\|f\|_{L^2(|x|^{b}dx)}
$$
for all $\alpha>0$, $\beta>0$ and $-d<b<d$. Thus it  suffices  to prove
$$
\Big\|\Big(\int_0^\infty \big|\cF^{-1} \big[\eta_1(\frac \rho t)
(1-\frac{\rho^\beta}{t^\beta})_+^{\alpha-1}
 \widehat f \,\big]\big|^2 \frac{dt}{t}\Big)^{1/2}
\Big\|_{L^2(|x|^{-2\alpha}dx)}
\lc \|f\|_{L^2(|x|^{-2\alpha}dx)}
$$
for $\eta_1\in C^\infty$ supported in $ (3/4,5/4)$.
 By Littlewood-Paley  theory for  the spaces $L^2(|x|^{-2\alpha}dx)$ and scaling it suffices to prove the analogous inequality
in the $t$-localized case  for which  in the last display
the integral over $\bbR^+$ is replaced by the integral over $[1,2]$.

We observe also that, with $\eta_2 \in C^\infty_c(0,\infty)$,
$$
\frac { \eta_2(\tfrac {\rho(\xi)}{t})(1-\tfrac{\rho(\xi)^\beta}{t^\beta})_+^{\alpha-1}}
{ (1-\tfrac{\rho(\xi)}{t})_+^{\alpha-1}} = g_\alpha(\xi/t)
$$
where  $g_\alpha\in \sL^2_\beta$ for all $\alpha>0$, $\beta>0$. Thus,
by an elementary inequality for vector valued operators
it suffices to prove that for $1/2<\alpha<d/2$
$$
\Big\|\Big(\int \big|\cF^{-1} \big[\eta(\rho)
 \varsigma(t)
(t-\rho)_+^{\alpha-1}
 \widehat f \,\big]\big|^2 dt\Big)^{1/2}
\Big\|_{L^2(|x|^{-2\alpha}dx)}
\lc \|f\|_{L^2(|x|^{-2\alpha}dx)}\,;
$$
here $\varsigma$ is a compactly supported $C^\infty$ function.
By applying Plancherel's theorem in $t$ the preceding inequality
follows from
\Be\label{GalphalocF}
\Big\|\Big(\int_{-\infty}^\infty (1+|\tau|^2)^{-\alpha}|\cA_\tau f|^2 d\tau \Big)^{1/2}
\Big\|_{L^2(|x|^{-2\alpha}dx)}
\lc \|f\|_{L^2(|x|^{-2\alpha}dx)}
\Ee
with $\cA_\tau $ as in \eqref{Atau}
As before let $\chi_l$ be the characteristic function of the annulus $\fA_l$.
Let $I_0=[-1,1]$ and, for $j\ge 1$,  $I_j=[-2^{j}, -2^{j-1}]\cup[2^{j-1}, 2^j]$.
In order to estimate
\eqref{GalphalocF}
we make a decomposition which is analogous to the one in \eqref{M12def}, and  prove
\begin{multline}
\label{hardytype}
\Big\|\Big(\sum_{j=0}^\infty 2^{-2j\a}
\int_{I_j}\big|\cA_\tau (\sum_{n\ge -c_0} f\chi_{j-n})\big|^2 d\tau \Big)^{1/2}
\Big\|_{L^2(|x|^{-2\a}dx)}\\ \lc
\|f\|_{L^2(|x|^{-2\alpha}dx)}\,,
\end{multline} and \begin{multline}
\label{L2type}
\Big\|\Big(\sum_{j=0}^\infty 2^{-2j\a}
\int_{I_j}\big|\cA_\tau (\sum_{l\ge j+c_0} f\chi_{l})\big|^2 d\tau \Big)^{1/2}
\Big\|_{L^2(|x|^{-2\a}dx)} \\ \lc
\|f\|_{L^2(|x|^{-2\alpha}dx)}\,.
\end{multline}
In order to show \eqref{hardytype}  we prove for $m\in \bbZ$, $n\ge -c_0$,
 $\eps <\min\{\alpha-\tfrac 12, \tfrac d2-\alpha\}$,
\Be\label{Gmn} \Big(\sum_{j=1}^\infty2^{-2j\a}\int_{I_j}
\big\|\chi_m \cA_\tau(f\chi_{j-n})
\big\|_{L^2(|x|^{-2\a}dx)}^2d\tau\Big)^{1/2} \lc 2^{-|m|\eps}
2^{-n\alpha} \|f\|_2\,. \Ee

For the proof of \eqref{Gmn}
let $1<b<d$. The left hand  side of \eqref{Gmn} can be estimated by
\begin{align*}& 2^{-m\frac{2\a -b}{2}}  \Big(\sum_j 2^{-2j\alpha}
\int
\big\|\cA_\tau (f\chi\ci{j-n})\big\|_{L^2(|x|^{-b}dx)}^2 d\tau\Big)^{1/2}
\\
&\lc 2^{-m\frac{2\a -b}{2}}  \Big(\sum_j 2^{-2j\alpha}
\big\|f\chi\ci{j-n}\big\|_{2}^2 \Big)^{1/2}\,,
\end{align*} by Proposition \ref{Atauprop}. The last display is estimated by
\begin{align*} &  2^{-m\frac{2\a -b}{2}} 2^{-n\alpha} \Big(\sum_j
\big\|f\chi\ci{j-n}\big\|_{L^2(|x|^{-2\a}dx)}^2 \Big)^{1/2}
\\
&\lc 2^{-m\frac{2\a -b}{2}} 2^{-n\alpha} \|f\|_{L^2(|x|^{-2\a}dx)}\,.
\end{align*}
We may choose $b$ such that
$1<b<2(\alpha-\eps)$ if $m>0$ and $2(\alpha+\eps)<b<d$ if $m\le 0$ and
 then \eqref{Gmn} follows.

We sketch the proof of the more straightforward inequality
\eqref{L2type} and rely on   the argument used for $\cM^2_k$ before.
Let $\widetilde \chi_l$ be the characteristic function of
$\cup_{-5\le i\le 5} \fA_{l+i}$. Then for $l>j+c_0$ and $\tau\in
I_j$
 $$(1-\widetilde \chi_{l})\cA_\tau(f\chi_l)
(x)\le C_N 2^{-lN}
\int\frac{|f(y)\chi_l(y)|}{
(1+2^l |x-y|)^{N} }dy
\,.$$ Since
$|x|^{-2\a}$ is an $A_2$ weight, it is immediate that
$$\Big(\sum_{j=0}^\infty 2^{-2j\a}\int_{I_j}
\Big\|\sum_{l>j} (1-\widetilde \chi_{l})|\cA_\tau(f\chi_l)|
 \Big\|_{L^2(|x|^{-2\a}dx)}^2 d\tau \Big)^{1/2}
 \lc
\|f\|_{L^2(|x|^{-2\alpha}dx)}.
$$

For the main term we use $\alpha>1/2$ and estimate
\begin{align*}
&\Big(\sum_{j=0}^\infty 2^{-2j\a}\int_{I_j}
\Big\|\sum_{l>j} \widetilde \chi_{l}|\cA_\tau(f\chi_l)|
\Big\|_{L^2(|x|^{-2\a}dx)}^2d\tau\Big)^{1/2}
\\ &\lc
\Big(\sum_{j=0}^\infty 2^{-2j\a}\int_{I_j}\sum_{l>j}2^{-2l\a}
 \|\cA_\tau(f\chi_l)\|_2^2\,d\tau
\Big)^{1/2}
\\ &\lc
\Big(\sum_{l>0}\sum_{j=0}^{l-1} 2^{-2j(\a-1)}2^{-2l\a}\|f\chi_l\|_2^2
\Big)^{1/2} \lc
\|f\|_{L^2(|x|^{-2\alpha}dx)}.
\end{align*}
Thus \eqref{L2type} is proved.\qed


\newpage
\end{document}